\documentclass[12]{amsart}
\usepackage{amsmath,amssymb,amsthm,centernot}
\usepackage{mathtools}

\DeclarePairedDelimiter\floor{\lfloor}{\rfloor}

\theoremstyle{plain}
\newtheorem{theorem}{Theorem}
\newtheorem{lemma}{Lemma}

\newtheorem{definition}{Definition}
\newtheorem{conjecture}{Conjecture}
\theoremstyle{definition}
\newtheorem{example}{Example}
\newtheorem{remark}{Remark}

\DeclareMathOperator{\lcm}{lcm}

\def\A {{\mathcal A}}
\def\L {{\mathcal L}}
\def\lcm {{\rm lcm}}
\def\C {{\Gamma}}
\def\CardCM {{\left|\Gamma_M\right|}}

\numberwithin{equation}{section}

\def\Z {{\mathbb Z}}
\begin{document}

\title[On Primitive Covering Numbers]{On Primitive Covering Numbers}


\author{Lenny Jones}
\address{Department of Mathematics, Shippensburg University, Pennsylvania, USA}
\email[Lenny~Jones]{lkjone@ship.edu}

\author{Daniel White}
\address{Department of Mathematics, Shippensburg University, Pennsylvania, USA}
\email[Daniel~White]{DWhite@ship.edu}


\date{\today}

\begin{abstract}
 In 2007, Zhi-Wei Sun defined a \emph{covering number} to be a positive integer $L$ such that there exists a covering system of the integers where the moduli are distinct divisors of $L$ greater than 1. A covering number $L$ is called \emph{primitive} if no proper divisor of $L$ is a covering number. Sun constructed an infinite set $\L$ of primitive covering numbers, and he conjectured that every primitive covering number must satisfy a certain condition. In this paper, for a given $L\in \L$, we derive a formula that gives the exact number of coverings that have $L$ as the least common multiple of the set $M$ of moduli, under certain restrictions on $M$. Additionally, we disprove Sun's conjecture by constructing an infinite set of primitive covering numbers that do not satisfy his primitive covering number condition.
\end{abstract}


\maketitle

\section{Introduction}\label{Section:Intro}
We begin with the definition of a concept due to Erd\H{o}s \cite{8}.
\begin{definition}\label{Def:Cov}
Let $x_i$ and $m_i$ denote integers, where $x_i\ge 0$ and $m_i\ge 2$. A {\it (finite) covering system} $C$, or simply a {\it covering}, of the integers is a finite collection of congruences $z\equiv x_i \pmod{m_i}$, such that every integer satisfies at least one of these congruences.
 \end{definition}
Throughout this paper we assume that all moduli in any covering  are distinct. We write a covering  as $C=\{(x_i,m_i)\}_{i\in I}$, where $z\equiv x_i \pmod{m_i}$ is a congruence in the covering, and $I$ is some finite indexing set. For $I=\{1,2,\ldots, t\}$, we let $M$ denote the set of moduli $\left\{m_1,m_2,\ldots ,m_t\right\}$ used in $C$, and we write $\lcm(M)$ for $\lcm(m_1,m_2,\ldots, m_t)$. We also let $\C_M$ denote the set of all coverings having $M$ as the set of moduli. The main focus in this article is on coverings and sets of moduli with the following special property.
\begin{definition}\label{Def:Minimal}
A covering $C$ is called \emph{minimal} if no proper subset of $C$ is a covering. Let $M$ be a set of positive integers for which $\C_M\ne \emptyset$. We say that $M$ is \emph{minimal} if all $C\in \C_M$ are minimal.
 \end{definition}
   The definition of a minimal set of moduli in Definition \ref{Def:Minimal}
 is not superfluous since there exist sets of moduli $M$ for which some elements of $\C_M$ are minimal and some are not. The following example illustrates this phenomenon.
 \begin{example}\label{Ex:QuasiMinimal}
 Consider the set
           \[M=\{3,4,5,6,8,10,12,15,20,24,30,40,60,120\}.\] Let
         \begin{align*}
         \begin{split}
          B=&\{(0,3),(0,4),(0,5),(1,6),(6,8),(3,10),(5,12),(11,15),\\
          &\quad (7,20),(10,24),(2,30),(34,40),(59,60),(98,120)\},
          \end{split}
          \end{align*}
          and
          \begin{align*}
          \begin{split}
          C=&\{(2,3),(0,4),(0,5),(3,6),(2,8),(7,10),(6,12),(1,15),\\
          &\quad (19,20),(22,24),(13,30),(0,40),(49,60),(0,120)\}.
          \end{split}
          \end{align*}
          It is straightforward to verify that $B$ 
           and $C$ 
            are coverings so that $B,C\in \C_M$. A bit more effort shows that $B$ is minimal. However, note that the elements $(0,40)$ and $(0,120)$ can be removed from $C$ and the remaining set $\widehat{C}$ is a covering; in fact, $\widehat{C}$ is minimal.

  The covering $B$ 
   is due to Erd\H{o}s \cite{8}, while the covering $\widehat{C}$ 
    is due to Krukenberg \cite{30}. We thank Mark Kozek for pointing these out to us.
  \end{example}


There are many situations when $M$ is minimal. For example,
\[C=\{(0,2), (0,3), (1,4), (1,6), (11,12) \}\] is a covering , but it is easy to see that it is impossible to construct a covering  using any proper subset of $M=\{2,3,4,6,12\}$. This example can be generalized to the situation when $M$ is the set of all divisors $d>1$ of $2^{p-1}p$, where $p>2$ is prime \cite{42}.

In 2007, Zhi-Wei Sun \cite{42} introduced the notion of a primitive covering number. A positive integer $L$ is called a \emph{covering number} if there exists a covering  of the integers where the moduli are distinct divisors $d>1$ of $L$. A covering number $L$ is called a \emph{primitive covering number} if no proper divisor of $L$ is a covering number.
Note that if $L$ is a primitive covering number, then there exists a set of moduli $M$ with $\lcm(M)=L$ such that $M$ is minimal.

In \cite{42}, Sun proved the following theorem, which we state without proof.

\begin{theorem}\label{Thm:Sun0}
  Let $p_1,p_2,\ldots ,p_r$ be distinct primes, and let $\alpha_1,\alpha_2,\ldots, \alpha_r$ be positive integers. Suppose that
 \begin{equation}\label{Eq:Sun}
  \prod_{0<t<s}(\alpha_t+1)\ge p_s-1+\delta_{r,s},\hspace*{.2in}   \mbox{for all $s=1,2,\ldots ,r$,}
\end{equation}
where $\delta_{r,s}$ is Kronecker's delta, and the empty product $\prod_{0<t<1}(\alpha_t+1)$ is defined to be 1.
Then $p_1^{\alpha_1}p_2^{\alpha_2}\cdots p_r^{\alpha_r}$ is a covering number.
\end{theorem}

 The following theorem, which we state without proof, is also due to Sun \cite{42}, and gives sufficient conditions for a positive integer to be a primitive covering number.
\begin{theorem}\label{Thm:Sun1}
  Let $r>1$ and let $2=p_1<p_2<\cdots <p_r$ be primes. Suppose further that $p_{t+1}\equiv 1 \pmod{p_t-1}$ for all $0<t<r-1$, and $p_r\ge (p_{r-1}-2)(p_{r-1}-3)$. Then
\[p_1^{\frac{p_2-1}{p_1-1}-1}\ldots p_{r-2}^{\frac{p_{r-1}-1}{p_{r-2}-1}-1}p_{r-1}^{\floor*{\frac{p_{r}-1}{p_{r-1}-1}}}p_r\] is a primitive covering number, where $\floor*{x}$ denotes the greatest integer less than or equal to $x$.
\end{theorem}

 Theorem \ref{Thm:Sun1} produces an infinite set $\L$ of primitive covering numbers, and every element of $\L$ satisfies \eqref{Eq:Sun}. 
  In this article, we derive a formula that gives the exact number of covering systems for each $L\in \L$, when the associated set of moduli $M$ is minimal. This represents the first such counting formula of its kind to appear in the literature. In addition, we construct an infinite set of primitive covering numbers that do not satisfy \eqref{Eq:Sun}, and thereby provide infinitely many counterexamples to a conjecture of Sun \cite{42}.
\section{The Number of Coverings for $L\in \L$}
Throughout this section, we let $\L$ denote the set of all primitive covering numbers that satisfy the conditions of Theorem \ref{Thm:Sun1}. We also let $C$ be a minimal covering with distinct moduli $M$, such that \[\lcm(M)=L= p_1^{\alpha_1} \cdots p_r^{\alpha_r} \in \L,\] where $2=p_1 < p_2 < \cdots < p_r$ are prime.

\begin{definition}\label{Def:C lambda} For each pair $(s,t)$ of integers with $1\le s \le r$ and $1\le t \le \alpha_s$, define
\begin{equation*}
N_{s,t}:=\{n\in \Z : 1 \le n \le p_1^{\alpha_1} \cdots p_{s-1}^{\alpha_{s-1}} p_s^t\},
\end{equation*} 
\begin{equation}\label{C_p}
C_{p_s^t} := \{ (*,m)\in C :\, p_s^t || m \,\,\mbox{ and } \,\,  p_k \nmid m \text{ for all } k > s \}
\end{equation}
\begin{equation}\label{lambda}
 \mbox{and} \quad \lambda_{s,t} := \left| \Big\{ n\in N_{s,t} : n \text{ is not covered by } \bigcup_{i=1}^{s-1} \bigcup_{j=1}^{\alpha_i} C_{p_i^j} \cup \bigcup_{j=1}^{t} C_{p_s^j} \Big\} \right|,
\end{equation}
where $\bigcup_{i=1}^{s-1} \bigcup_{j=1}^{\alpha_i} C_{p_i^j}=\emptyset$ if $s=1$.
\end{definition}

Observe that $\lambda_{r,\alpha_r}=0$ in \eqref{lambda}, and that \eqref{C_p} implies that
\[
C = \bigcup_{i=1}^r \bigcup_{j=1}^{\alpha_i} C_{p_i^j}.
\]

\begin{lemma} \label{lem:hole1}
If $\lambda_{s,\alpha_s}=1$, then $\lambda_{s+1,1} = p_{s+1} - |C_{p_{s+1}}|$.
\end{lemma}

\begin{proof}
Let
\[\A=\left\{ n\in N_{s+1,1} : n \text{ is not covered by } \bigcup_{i=1}^s \bigcup_{j=1}^{\alpha_i} C_{p_i^j} \right\}.\]
Supposing $\lambda_{s,\alpha_s}=1$, we see that $\left|\A\right|=p_{s+1}$.
In addition, each $n\in \A$ is in a unique congruence class modulo $p_{s+1}$. Hence, each element of $C_{p_{s+1}}$ may cover at most one $n\in \A$. Since $C$ is minimal, each element of $C_{p_{s+1}}$ covers at least one $n\in \A$, and thus the proof is complete.
\end{proof}

\begin{lemma} \label{lem:hole2}
If $\lambda_{s,t}=1$ and $0 < t < \alpha_s$, then $\lambda_{s,t+1} = p_s - |C_{p_s^{t+1}}|$.
\end{lemma}
The proof of Lemma \ref{lem:hole2} follows an argument similar to the proof of Lemma \ref{lem:hole1} and is omitted.


\begin{lemma} \label{Lem:Dsubset}
 Let \[L=p_1^{\alpha_1}p_2^{\alpha_2}\cdots p_{r-1}^{\alpha_{r-1}}p_r\in \L,\] so that the $\alpha_i$ satisfy the conditions of Theorem \ref{Thm:Sun1}. Let $M$ be minimal with $\lcm(M)=L$, and let
\[D=\left\{d>1: d\mid L/(p_{r-1}p_r)\right\}.\] Then $D\subseteq M$.
\end{lemma}

\begin{proof}
Let $C\in \C_M$.
It suffices to prove that $C$ is not a covering if $M$ is missing exactly one element from $D$. So suppose that $M$ is missing exactly one element $m^{\prime}$ from $D$ and write \[m' = p_1^{j_1} \cdots p_{r-1}^{j_{r-1}},\] where $0 \le j_i \le \alpha_i$ and $j_{r-1} < \alpha_{r-1}$. If $m' \neq p_1^{\alpha_1} \cdots p_{r-1}^{\alpha_{r-1}-1}$ and we remove the congruence
\[\left(*,p_1^{\alpha_1} \cdots p_{r-1}^{\alpha_{r-1}-1}\right)\] from $C$ 
(whatever $*$ may be), then at least one integer in $[1,L]$ is not covered. In particular, each uncovered integer falls into the congruence class removed. We can then add the congruence $\left(*,m^{\prime}\right)$
to $C$, which covers the integers in consideration. This alteration to $C$ provides another covering. Hence, it suffices to prove that $C$ is not a covering when
\begin{equation}\label{Eq:m'} m^{\prime}=p_1^{\alpha_1} \cdots p_{r-1}^{\alpha_{r-1}-1} = L/(p_{r-1} p_r).
\end{equation}

Observe that
\begin{equation}\label{Eq:1}
\left|C_{p_s^t}\right| = \left\{
\begin{array}{cl}
1 & \qquad \mbox{if $s=1$}\\
\\
\displaystyle  \prod_{i=1}^{s-1} \frac{p_{i+1}-1}{p_i-1} = p_s-1 & \qquad \mbox{if $1<s<r-1$ and $1\le t\le \alpha_s$}\\
& \qquad \mbox{ or $s=r-1$ and $1 \le t < \alpha_{s}-1$}.
\end{array} \right.
\end{equation}
Thus, by Lemma \ref{lem:hole2} and \eqref{Eq:1}, we have that $\lambda_{1,t} = 1$ for each $1 \le t \le \alpha_1$. Additionally, \eqref{Eq:1}, in conjunction with Lemma \ref{lem:hole1}, implies that $\lambda_{2,1}=1$, and hence $\lambda_{2,t}=1$ for each $1 \le t \le \alpha_2$. Continuing in this manner, we see that $\lambda_{r-1,\alpha_{r-1}-2}=1$, and thus, $\lambda_{r-1,\alpha_{r-1}-1}=2$ by  \eqref{Eq:m'}. 
Hence,
$\left|\A\right|=2p_{r-1}$,
where
\[\A=\left\{ n\in N_{r-1,\alpha_{r-1}} : n \text{ is not covered by } \bigcup_{i=1}^{r-2} \bigcup_{j=1}^{\alpha_i} C_{p_i^j} \, \cup \bigcup_{j=1}^{\alpha_{r-1}-1} C_{p_{r-1}^j}
\right\},\]
and exactly half of the uncovered integers in $\A$ 
 are in the same congruence class modulo $p_{r-1}^{\alpha_{r-1}-1}$, and the other half in another single congruence class modulo $p_{r-1}^{\alpha_{r-1}-1}$. 
 If we assume, at best, that
\[
C_{p_{r-1}^{\alpha_{r-1}}} = \left\{(*,p_{r-1}^{\alpha_{r-1}} k)\in C : k \mid p_1^{\alpha_1} \cdots p_{r-2}^{\alpha_{r-2}} \right\},
\]
then $\lambda_{r-1,\alpha_{r-1}} = p_{r-1}+1$. Additionally, some uncovered integer in $[1,L/p_r]$ is in a different congruence class modulo $p_{r-1}^{\alpha_{r-1}-1}$ than another. If we assume, minimally, that $\lambda_{r-1,\alpha_{r-1}}=3$, then exactly two uncovered integers will be in the same congruence class modulo $p_{r-1}^{\alpha_{r-1}-1}$.

Under this assumption, we get
\begin{equation}\label{Eq:3p_r}
\left| \left\{ n\in N_{r,1} : n \text{ is not covered by } \bigcup_{i=1}^{r-1} \bigcup_{j=1}^{\alpha_i} C_{p_i^j}\right\} \right| = 3p_r,
\end{equation}
which implies that at least $3p_r$ integers in $[1,L]$ not covered by
$\cup_{i=1}^{r-1}\cup_{j=1}^{\alpha_s}C_{p_i^j}$ need to be covered by $C_{p_r}$.

On the other hand, we have at best that
\[
C_{p_r} = \left\{ (*,p_r k)\in C : k \mid L/p_r \right\}.
\]
Now, for $c=(*,m)\in C_{p_r}$, there are three possibilities.
\begin{enumerate}
\item If $p_{r-1}^{\alpha_{r-1}-1} \nmid m$, then $c$ can cover up to three uncovered integers in $[1,L]$.
\item If $p_{r-1}^{\alpha_{r-1}-1} \mid \mid m$, then $c$ can cover up to two uncovered integers in $[1,L]$.
\item If $p_{r-1}^{\alpha_{r-1}} \mid \mid m$, then $c$ can cover up to one uncovered integer in $[1,L]$.
\end{enumerate}
Using this information to construct an upper bound on the number of uncovered integers in $[1,L]$ that can be covered by $C_{p_r}$, we get:
\begin{align}\label{Eq:Upper}
\begin{split}
3 \tau \left( \frac{L}{p_r p_{r-1}^2} \right) & + 2 \tau \left( \frac{L}{p_r p_{r-1}^{\alpha_{r-1}}} \right) + \tau \left( \frac{L}{p_r p_{r-1}^{\alpha_{r-1}}} \right) \\
& = 3 \left\{ \tau \left( \frac{L}{p_r p_{r-1}^{\alpha_{r-1}}} \right) + \tau \left( \frac{L}{p_r p_{r-1}^2} \right) \right\} \\
& = 3 \left\{ (p_{r-1}-1) + (p_{r-1}-1) \left( \floor*{\frac{p_r-1}{p_{r-1}-1}} -1 \right) \right\} \\
& \le 3(p_r-1),
\end{split}
\end{align}
where $\tau(z)$ is the number of divisors of $z$. From \eqref{Eq:3p_r}, we see that \eqref{Eq:Upper} contradicts the fact that $C$ is a covering and completes the proof of the lemma.
\end{proof}

\begin{theorem}\label{Thm:Main1}
 Let $L=p_1^{\alpha_1}p_2^{\alpha_2}\cdots p_{r-1}^{\alpha_{r-1}}p_r^{\alpha_r}$. Suppose that $L\in \L$, so that $\alpha_r=1$ and the $\alpha_i$ satisfy the conditions of Theorem \ref{Thm:Sun1}. Let $M$ be minimal with $\lcm(M)=L$. Then 
\begin{equation}\label{Eq:Formula}
\CardCM =
\frac{\left(\left|C_{p_r}\right|-\left|Q\right|\right)!}{\left(p_{r-1}-\left|C_{p_{r-1}^{\alpha_{r-1}}}\right|\right)!\left(p_r-\left|Q\right|\right)!} \prod_{i=1}^{r} \left(p_i !\right)^{\alpha_i},
\end{equation}
where $Q := \left\{ (*,m) \in C_{p_r} : p_{r-1}^{\alpha_{r-1}} \nmid m \right\}$.
\end{theorem}

\begin{proof}
As in the proof of Lemma \ref{Lem:Dsubset}, we have that $\lambda_{s,t} = 1$ for $1 \le t \le \alpha_i$ when $1 \le s < r-1$, and for $1 \le t < \alpha_{r-1}$ when $s=r-1$. Hence, for such $t$, there are exactly $p_i!$ ways to choose the residues for the congruences in $C_{p_i^t}$ assuming that
\[\begin{array}{c}
C_{p_1}, C_{p_1^2}, \ldots,  C_{p_1^{\alpha_1}}, C_{p_2}, C_{p_2^2}, \ldots,  C_{p_2^{\alpha_2}}, \ldots, C_{p_i}, C_{p_i^2},\ldots, C_{p_i^{t-1}}\\
 \text{ or} \\
C_{p_1}, C_{p_1^2}, \ldots, C_{p_1^{\alpha_1}}, C_{p_2}, C_{p_2^2}, \ldots, C_{p_2^{\alpha_2}}\ldots, C_{p_{i-1}}, C_{p_{i-1}^2}, \ldots, C_{p_{i-1}^{\alpha_{i-1}}}
\end{array}\]

are already determined, if $t>1$ or if $t=1$, respectively. Inductively, we have that there are
\begin{equation}\label{eqn:first}
(p_1!)^{\alpha_1} \cdots (p_{r-2}!)^{\alpha_{r-2}} (p_{r-1}!)^{\alpha_{r-1}-1}
\end{equation}
ways to construct the set
\begin{equation}\label{eqn:set}
\left\{ C_{p_{r-1}}, C_{p_{r-1}^2}, \ldots ,C_{p_{r-1}^{\alpha_{r-1}-1}}\right\} \cup \bigcup_{i=1}^{r-2} \left\{ C_{p_{i}}, C_{p_{i}^2}, \ldots ,C_{p_{i}^{\alpha_i}}\right\}.
\end{equation}
When constructing $C_{p_{r-1}^{\alpha_{r-1}}}$, we must first choose which integers in $[1,L/p_r]$ to cover that are not already covered by (\ref{eqn:set}). Since there are $p_{r-1}$ such integers, each in a difference congruence class modulo $p_{r-1}^{\alpha_{r-1}}$, there are exactly
\begin{equation}\label{eqn:second}
\begin{pmatrix}
  p_{r-1}\\
  \left|C_{p_{r-1}^{\alpha_{r-1}}}\right|
\end{pmatrix}
\cdot \left|C_{p_{r-1}^{\alpha_{r-1}}}\right|! = \frac{p_{r-1}!}{\left(p_{r-1} - \left|C_{p_{r-1}^{\alpha_{r-1}}}\right|\right)!}
\end{equation}
ways to choose the residues for the congruences in $C_{p_{r-1}^{\alpha_{r-1}}}$. At this point, there are
\[
p_r \lambda_{r-1,\alpha_{r-1}} = p_r\left(p_{r-1}-\left|C_{p_{r-1}^{\alpha_{r-1}}}\right|\right)
\]
integers uncovered in $[1,L]$. In particular, there are $p_{r-1}-\left|C_{p_{r-1}^{\alpha_{r-1}}}\right|$ uncovered integers in each congruence class modulo $p_r$. Finally, we construct $C_{p_r}$ in two steps. First, we consider the congruences which will reside in $Q$. Each $(*,m) \in Q$ will cover all $p_{r-1}-\left|C_{p_{r-1}^{\alpha_{r-1}}}\right|$ of these uncovered integers in one of the congruence classes modulo $p_r$. Hence, we choose the congruence class modulo $p_r$, and then the congruence that covers the class. There are
\begin{equation}\label{eqn:third}
{p_r \choose |Q|} \cdot |Q|! = \frac{p_r!}{(p_r - |Q|)!}
\end{equation}
ways to do this. The remaining uncovered integers in $[1,L]$ must be covered by $C_{p_r} \setminus Q$. Each congruence in $C_{p_r} \setminus Q$ covers exactly one of the remaining uncovered integers in $[1,L]$. There are
\begin{equation}\label{eqn:fourth}
(|C_{p_r}|-|Q|)!
\end{equation}
ways to choose the residues for these congruences. Using (\ref{eqn:first}), (\ref{eqn:second}), (\ref{eqn:third}), and (\ref{eqn:fourth}), we conclude that
\begin{align*}
\CardCM &=
\prod_{i=1}^{r-2} \left(p_i!\right)^{\alpha_i} \cdot \left(p_{r-1}!\right)^{\alpha_{r-1}-1} \cdot \frac{p_{r-1}!}{\left(p_{r-1} - \left|C_{p_{r-1}^{\alpha_{r-1}}}\right|\right)!} \cdot \frac{p_r!}{\left(p_r - |Q|\right)!} \cdot \left(|C_{p_r}|-|Q|\right)!\\
&=\frac{\left(\left|C_{p_r}\right|-\left|Q\right|\right)!}{\left(p_{r-1}-\left|C_{p_{r-1}^{\alpha_{r-1}}}\right|\right)!\left(p_r-\left|Q\right|\right)!} \prod_{i=1}^{r} \left(p_i !\right)^{\alpha_i},
\end{align*}
and the proof is complete.
\end{proof}

\begin{remark}
For many ``small" values of $L\in \L$, formula \eqref{Eq:Formula} reduces to
\[\CardCM=\prod_{i=1}^{r} \left(p_i !\right)^{\alpha_i}.\] For example, straightforward calculations show that
\[M=\{2, 4, 5, 8, 10, 16, 20, 40, 80\}\]
is minimal and $L=2^4\cdot 5\in \L$. In this case, we have that $p_r=p_2=5$, $p_{r-1}=p_1=2$,
\[p_r=\left|C_{p_r}\right| \quad \mbox{and} \quad p_{r-1}=\left|C_{p_{r-1}^{\alpha_{r-1}}}\right|,\]
so that
\[\CardCM=\left(2!\right)^4\cdot \left(5!\right)=1920.\]
\end{remark}
\section{Counterexamples to a Conjecture of Sun}
In \cite{42}, Sun made the following conjecture.
\begin{conjecture}\label{Conj:Sun}
  Any primitive covering number can be written as $p_1^{a_1}\cdots p_r^{a_r}$, where $p_1,\ldots,p_r$ are distinct primes and $a_1,\ldots ,a_r$ are positive integers that satisfy \eqref{Eq:Sun}.
\end{conjecture}

Conjecture \ref{Conj:Sun} is false and the following theorem provides infinitely many counterexamples. We let $q_n$ denote the $n$th prime number.

\begin{theorem}\label{Thm:SunCounterEx}
For any $\delta\in \Z$, with $\delta\ge 3$, there exist infinitely many primitive covering numbers of the form $2^{\beta} q_k q_{k+1}$, where $\beta\le q_k-3$.  Consequently, each such primitive covering number fails to satisfy \eqref{Eq:Sun} and provides a counterexample to Conjecture \ref{Conj:Sun}.
\end{theorem}
\begin{proof}
  Let $\delta\in \Z$ with $\delta\ge 3$. By the prime number theorem, there exist infinitely many primes $q_k$ such that
\begin{equation}\label{Eq:PNT}
\dfrac{q_k}{q_{k+1}}\ge \dfrac{\delta-1}{\delta}+\dfrac{\delta-1}{q_{k+1}}.
\end{equation}
Rewriting \eqref{Eq:PNT}, we have that
\begin{equation}\label{Eq:MH}
q_k-\delta +1\ge \left(\delta-1\right)\left(q_{k+1}-\left(q_k-\delta+1\right)\right).
\end{equation}
Let $q_k$ be a prime that satisfies \eqref{Eq:MH}, and let $L=2^{q_k-\delta}q_kq_{k+1}$. Note that $L$ does not satisfy \eqref{Eq:Sun}. We claim that $L$ is a covering number. To establish the claim, we construct a covering using the divisors of $L$ in the following way. We first use the moduli
\[2,\quad 2^2, \quad 2^3, \ldots, 2^{q_k-\delta}.\] This leaves one hole modulo $2^{q_k-\delta}$. We introduce the prime $q_k$ to split this single hole into $q_k$ holes. We can fill $q_k-\delta+1$ of these holes using the moduli
\[q_k,\quad 2q_k,\quad 2^2q_k,\ldots, 2^{q_k-\delta}q_k.\]  Denote the remaining $q_k-\left(q_k-\delta+1\right)=\delta-1$ holes as $A_1, A_2, A_3, \ldots, A_{\delta-1}$.

  Now we introduce the prime $q_{k+1}$, and we split each $A_i$ into $q_{k+1}$ holes. For each $i$, we can use the moduli
\[q_{k+1}, \quad 2q_{k+1}\quad 2^2q_{k+1},\ldots, 2^{q_k-\delta}q_{k+1}\]
to fill $q_k-\delta+1$ of the $q_{k+1}$ holes. Therefore, at this point, we have a total of
\[\left(\delta-1\right)\left(q_{k+1}-\left(q_k-\delta+1\right)\right)\]
holes left to fill. However, we still have the $q_k-\delta+1$ unused moduli
\[q_kq_{k+1},\quad 2q_kq_{k+1},\quad  2^2q_kq_{k+1},\ldots, 2^{q_k-\delta}q_kq_{k+1}.\]
Hence, by \eqref{Eq:MH}, we have established that $L$ is a covering number.

If $L$ itself is primitive, then $L$ provides a counterexample to Conjecture \ref{Conj:Sun}. So suppose that $L$ is not primitive. It is easy to see that the only divisors of $L$ that are candidates for covering numbers are of the form $D_{\beta}=2^{\beta}q_kq_{k+1}$, where $\beta<q_k-\delta$. Therefore, some proper divisor $D_{\beta}$ of $L$ is a primitive covering number. Note that $D_{\beta}$ also fails to satisfy \eqref{Eq:Sun}, and the proof is complete.
\end{proof}

We provide three concrete examples that arise from Theorem \ref{Thm:SunCounterEx} with $\delta=3$. In each example, we give a covering $C_i$ using divisors $d>1$ of $L_i$.
\begin{example}\label{Ex:1} $L_1=2^8\cdot 11\cdot 13$\\\label{Section:L1}
Let
\begin{align*}
  C_1&=\{(1,2), (0,4), (2,8), (0,11), (2,13), (6,16), (20,22), (20,26), (30,32),  (34,44),\\
  & \qquad (6,52), (46,64), (14,88), (38,104), (14,128), (138,143), (62,176), (30,208),\\
  & \qquad  (78,256), (94,286), (142,352), (206,416), (226,572), (654,704), (14,832), \\
  & \qquad (1062,1144), (206,1408), (78, 1664), (1214,2288), (2510, 2816), (2766, 3328),\\
  & \qquad  (334, 4576), (4558,9152), (7374,18304)\}.
\end{align*}
\end{example}

\begin{example}\label{Ex:2} $L_2=2^{14}\cdot 17\cdot 19$\\ \label{Section:L2}
Let
\begin{align*}
  C_2&=\{(1, 2), (0, 4), (2, 8), (14, 16), (0, 17), (1, 19), (22, 32), (14, 34), (8, 38),\\
   & \qquad (6, 64), (54, 68), (38, 76), (102, 128), (94, 136), (6, 152), (38, 256),\\
   & \qquad (70, 272), (262, 304), (245, 323), (422, 512), (486, 544), (230, 608),\\
   & \qquad (678, 1024), (358, 1088), (358, 1216), (1078, 1292), (1190, 2048),\\
   & \qquad (934, 2176), (1318, 2432), (1662, 2584), (2214, 4096), (4262, 4352),\\
   & \qquad (4774, 4864), (2070, 5168), (166, 8192), (2214, 8704), (678, 9728),\\
   & \qquad (12454, 16384), (8358, 17408), (15526, 19456),  (32934, 34816),\\
   & \qquad (30886, 38912), (2982, 41344), (28838, 69632), (41126, 77824),\\
   & \qquad (20646, 139264), (102566, 155648), (95910, 165376), (200870, 278528),\\
   & \qquad (20646, 311296), (129190, 330752), (166, 661504)\}.
 \end{align*}
 \end{example}
\begin{example}\label{Ex:3} $L_3=2^{16}\cdot 19\cdot 23$\\\label{Section:L3}
Let
\begin{align*}
  C_3&=\{(1, 2), (0, 4), (2, 8), (6, 16), (4, 19), (13, 23), (14, 32), (16, 38), (24, 46),\\
   & \qquad (62, 64), (62, 76), (18, 92), (30, 128), (150, 152), (30, 184), (222, 256),\\
   & \qquad (238, 304), (238, 368), (350, 512), (382, 608), (318, 736), (94, 1024),\\
   & \qquad (286, 1216), (1438, 1472), (1630, 2048), (478, 2432), (1758, 2944),\\
   & \qquad (1510, 3496), (606, 4096), (4446, 4864), (4190, 5888), (5822, 6992),\\
   & \qquad (6750, 8192), (3166, 9728), (5726, 11776), (2654, 16384), (11870, 19456),\\
   & \qquad (7774, 23552), (15646, 27968), (10846, 32768), (29278, 38912),\\
   & \qquad (25182, 47104), (59998, 65536), (23134, 77824), (2654, 94208),\\
   & \qquad (29534, 111872), (141918, 155648), (182878, 188416), (72798, 223744),\\
   & \qquad (59998, 311296), (207454, 376832), (326238, 447488), (158302, 622592),\\
   & \qquad (289374, 753664), (676446, 894976), (813662, 1245184), (813662, 1507328),\\
   & \qquad (977502, 1789952), (404062, 3579904), (1288798, 7159808),\\
   & \qquad (12544606, 14319616), (9071198, 28639232)\}.
 \end{align*}
 \end{example}
We note that in each of the previous examples not all divisors $d>1$ of $L_i$ are used as moduli in $C_i$. In fact, it is easy to show that equality cannot hold in \eqref{Eq:MH} for any value of $\delta$.

\end{document}